\documentclass[12pt]{amsart}
\usepackage{amsmath,amsfonts,amsthm,amsopn,cite,mathrsfs}

\newlength{\fixboxwidth}
\def\inv{^{-1}}
\def\rd{\mathbb{R} ^d}
\newtheorem{theorem}{Theorem}[section]
\newtheorem{lemma}[theorem]{Lemma}
\newtheorem{proposition}[theorem]{Proposition}
\newtheorem{corollary}[theorem]{Corollary}
\newtheorem{definition}[theorem]{Definition}
\newtheorem{remark}[theorem]{Remark}
\newtheorem{example}[theorem]{Example}

\begin{document}
\begin{abstract}
We study the notion of molecules in coorbit spaces. 
The main result states that if  an operator, originally defined on an
appropriate  space of test functions, 
maps atoms to molecules, then it can be extended to a bounded operator
on coorbit spaces. 
For time-frequency molecules we recover some boundedness results on
modulation spaces, 
for time-scale molecules we obtain the boundedness on homogenous Besov spaces.
\end{abstract}

\title{Molecules in Coorbit Spaces and Boundedness of Operators} %%
\author{Karlheinz Gr\"ochenig}
\address{Faculty of Mathematics \\
University of Vienna \\
Nordbergstrasse 15 \\
A-1090 Vienna, Austria}
\email{karlheinz.groechenig@univie.ac.at}
\author{Mariusz Piotrowski}
\email{mariusz.piotrowski@univie.ac.at}

\date{}
\thanks{Both authors were supported by the
Marie-Curie Excellence Grant MEXT-CT 2004-517154.}
\maketitle

%%%%%%%%%%%%%%%%%%%%%%%%%%%%%%%%%%%%%%%%%%%%%%%%%%%%%%%%%%%%%ABSTRACT
%%%%%%%%%%%%%%%%%%%%%%%%%%%%%%%%%%%%%%%%%%%%%%%%%%%%%%%%%%%%%%%%%%%%%%
%%%%%%%%%%%%%%%%%%%%%%%%%%%%%%%%%%%%%%%%%%%%%%%%%%%%%%%%%%%%%%%%%%%%%%%
%%%%%%%%%%%%%%%%%%%%%%%%%%%%%%%%%%%%%%%%%%%%%%%%%%%%%%%%%%%%%%%%%%%%%%%

%%%%%%%%%%%%%%%%%%%%%%%%%%%%%%%%%%%%%%%%%%%%%%%%%%%%%%Section_1
%%%%%%%%%%%%%%%%%%%%%%%%%%%%%%%%%%%%%%%%%%%%%%%%%%%%%%%%%%%%%%%
\section{Introduction}\label{section_1}
%%%%%%%%%%%%%%%%%%%%%%%%%%%%%%%%%%%%%%%%%%%%%%%%%%%%%%%%%%%%%%%
%%%%%%%%%%%%%%%%%%%%%%%%%%%%%%%%%%%%%%%%%%%%%%%%%%%%%%%%%%%%%%%
A remarkable principle of classical analysis states that an operator
that maps atoms to molecules is bounded. Here an ``atom'' is a function on
$\rd $ satisfying certain support and moment conditions, and  norm
bounds. Atoms arose first in the study of atomic decomposition of real
Hardy spaces~\cite{Coi74} and singular integral operators on Hardy
spaces (see~\cite{grafakos,stein93}). The notion of an atom was later
diversified to  adapt  to
 the Besov-Triebel-Lizorkin spaces~\cite{FJ90}, and then
generalized to ``molecules'', which are functions satisfying norm
bounds, moment and decay conditions (instead of support conditions),
see~\cite{CRTMW80,FJ90}.   
The resulting  molecular decompositions of function spaces  have been
successfully applied to  study  the boundedness properties of 
Calder\`on--Zygmund operators on Besov and Triebel--Lizorkin spaces,
see \cite{FrJa85,FJ90,FrJaWe91_Book,FHJW,Tor91} for some of the main contributions. 
The  technical part  of the proofs consists of showing that the
operator under consideration maps smooth atoms into smooth
molecules. Using norm estimates for atomic and molecular
decompositions, one then obtains the boundedness of the operator.  
A similar  strategy has been used in \cite{GrTo99} to study a class of
 pseudodifferential operators on Besov and Triebel--Lizorkin spaces.

In this paper we study atoms, molecules, and the  boundedness of operators
in the context of coorbit theory. In coorbit theory one can attach to
every irreducible, unitary, integrable representation $\pi $ of a
locally compact group $\mathcal{G}$ on a Hilbert space $\mathcal{H}$ a class
of $\pi $-invariant Banach spaces $\mathrm{Co}Y$ that is parametrized
by  function spaces $Y$ on the group
$\mathcal{G}$. These so-called coorbit spaces possess a rich theory
ranging from interpolation properties and duality theory to atomic
decompositions and the existence of frames, see the series of papers
\cite{FeGr88,FeGr89_1, FeGr89_2,Gr91}. The best known examples of
coorbit spaces are the Besov-Triebel-Lizorkin spaces (by choosing  the group
of affine transformations on $\rd $ and the representation by
translations and dilations) and the modulations spaces (by choosing
the Heisenberg group and the Schr\"odinger representation). An
interesting recent example is  the family of  shearlet spaces of~\cite{DJST08}. 

In the context of general coorbit spaces, the atoms are 
subsets $\{\pi (x_i)g: x_i \in \mathcal{G}\}$ in the orbit of the
representation $\pi $ for suitable $g\in \mathcal{H}$. One of the main
results of coorbit theory 
establishes the existence of atomic decompositions with respect to
such atoms~\cite{FeGr89_1,Gr91}.  
In the standard examples, these atomic decompositions imply the 
non-orthogonal wavelet expansions of  the  homogenous Besov spaces and
the      Gabor-type expansions of  the modulation spaces.

Our contribution here is the introduction of molecules in  general
  coorbit spaces and the study of their properties. Roughly
speaking, a set of molecules is determined  by an envelope function
$H$ on  the group $\mathcal{G}$ and a discrete subset of positions $\{x_i\}
$ in $ \mathcal{G}$. See Section~3 for the precise definition. 
Our main result then shows that any operator that maps a set of  atoms $\pi
(x_i)g$ to a set of  molecules is bounded on the associated  coorbit
spaces (Theorem~\ref{main_theorem}).

We then investigate what the abstract theorem says for the concrete
examples of the Schr\"odinger representation of the Heisenberg group
and for the group of affine transformations. For the Heisenberg group
we recover the notion of time-frequency molecules which were
introduced already  in ~\cite{bacahela06-1,Gr04,GrRz07}. Our main
theorem implies  
the boundedness  of pseudodifferential operators on  
modulation spaces~\cite{Gr06,GrRz07}. The use   of
time-frequency molecules  sheds a new light on
mapping properties of pseudodifferential operators.
 For the group of affine transformations we 
investigate explicit  time-scale molecules. Our main insight shows that   classical
smooth molecules are also  time-scale
molecules in the sense of coorbit theory.  As an effortless application of our  main result, we verify 
the boundedness of the Hilbert transform on homogenous Besov spaces.

The paper is organized as follows. In  Section~2  we summarize some of
the standard facts of coorbit space theory  from
\cite{FeGr89_1,Gr91}. 
We   recall the necessary 
definitions of function spaces on locally compact groups  and of
coorbit spaces,   and then  describe their atomic decompositions and
Banach frames. 
In Section~3 we introduce the notion of molecules in the context of coorbit spaces and study their fundamental properties.
 This section contains  our  main result about the boundedness of
 operators acting on coorbit spaces: if an operator  maps atoms to
 coorbit molecules, then it  can be
extended to a bounded operator on the  corresponding coorbit spaces.
Section~4 is devoted to make explicit the abstract theory for the
case of the 
Heisenberg group and of  the group of affine transformations.

%%%%%%%%%%%%%%%%%%%%%%%%%%%%%%%%%%%%%%%%%%%%%%%%%%%%%%%Section_2
%%%%%%%%%%%%%%%%%%%%%%%%%%%%%%%%%%%%%%%%%%%%%%%%%%%%%%%%%%%%%%%%
\section{Coorbit Space Theory}
%%%%%%%%%%%%%%%%%%%%%%%%%%%%%%%%%%%%%%%%%%%%%%%%%%%%%%%%%%%%%%%%
%%%%%%%%%%%%%%%%%%%%%%%%%%%%%%%%%%%%%%%%%%%%%%%%%%%%%%%%%%%%%%

First we recall the concepts and required results from the theory of
coorbit spaces. We work with functions spaces and representations on a
locally compact group. 

%%%%%%%%%%%%%%%%%%%%%%%%%%%%%%%%%%%%%%%%%%%%%%%%%%%%%%%Subsection_2.1
%%%%%%%%%%%%%%%%%%%%%%%%%%%%%%%%%%%%%%%%%%%%%%%%%%%%%%%%%%%%%%%%
\subsection{Preliminaries and notation}
%%%%%%%%%%%%%%%%%%%%%%%%%%%%%%%%%%%%%%%%%%%%%%%%%%%%%%%%%%%%%%%%
%%%%%%%%%%%%%%%%%%%%%%%%%%%%%%%%%%%%%%%%%%%%%%%%%%%%%%%%%%%%%%%%

In the sequel, let $\mathcal{G}$ be a locally compact group with identity $e$.
Integration on $\mathcal{G}$ will always
be with respect to the left Haar measure, and $\Delta$ is the  the
Haar modulus on $\mathcal{G}$. We denote by $L_x F(y) = F(x^{-1}y)$ and
$R_x F(y) = F(yx)$, $x,y \in \mathcal{G}$, the operators of left and
right translation. 
Further, we also need the involution $F^\vee(x) = F(x^{-1})$. %  and
% $F^\nabla(x) = \overline{F(x^{-1})}$.
The space of all bounded functions on $\mathcal{G}$ with compact support will be denoted by $L^{\infty}_0(\mathcal{G})$.
Let $\chi_U$ be the characteristic function of the set $U$.

%%%%%%%%%%%%%%%%%%%%%%%%%%%%%%%%%%%%%%%%%%%%%%%%%%%%%%%%%%%%%%%%%%%%%%%Subsection_2.2
\subsection{Banach Function Spaces on $\mathcal{G}$}\label{subsection_Banach_function_spaces}
%%%%%%%%%%%%%%%%%%%%%%%%%%%%%%%%%%%%%%%%%%%%%%%%%%%%%%%%%%%%%%%%%%%%%%%%%%%%%%%%%

 We work in the context of  Banach function spaces. We assumed that
 $Y$ is a Banach space consisting of functions on
 $\mathcal{G}$ equipped with the norm $\|\cdot|Y\|$ and that $Y$ satisfies
 the following properties. 
\begin{itemize}
  \item [(i)] $Y$ is continuously embedded into $L_{\mathrm{loc}}^1(\mathcal{G})$, the locally integrable functions on $\mathcal{G}$.
  \item [(ii)] $Y$ is {\em solid}, i.e., 
if $F \in Y$, $G$ is measurable and satisfies $|G(x)| \leq |F(x)|$ a.e., then
$G \in Y$ and $\|G|Y\| \leq \|F|Y\|$.
  \item [(iii)]$Y$ is invariant under left and right translations, i.e
    $L_xY\subseteq  Y$ and $R_xY\subseteq  Y$ for all $x\in \mathcal{G}$.
 If we denote $u(x)=|\!|\!| L_x|Y|\!|\!|$ and $v(x)=\Delta(x^{-1})|\!|\!|R_{x^{-1}}|Y|\!|\!|$,  the operator norms of translations on $Y$,
 then we require that
\[L^1_u * Y \subseteq Y \quad\text{and}\quad   Y*L^1_v \subseteq Y.\]
\end{itemize}
 We only  work with pairs $(Y,w)$, where the  weight function $w$  on $\mathcal{G}$  satisfies
\begin{align}\label{weight_w}
w(x)&\geq C \max \{u(x), u(x^{-1}),v(x), \Delta(x^{-1})v(x^{-1})\},\\
w(x)& =w(x^{-1}) \Delta(x^{-1})\nonumber
\end{align}
for some  constant $C>0$. 
In particular, $w(x)\geq 1$, $\|f|L^1_w\|=\|f^\vee |L^1_w\|$ and  $Y*L^1_w \subset  Y$.

We emphasize that the assumptions in coorbit theory concern mostly the
weight $w$ associated to $Y$, the main results hold simultaneously for
the entire class of function spaces $Y$  with the same weight $w$, and
not just for an single $Y$.

The Lebesgue spaces $L^p(\mathcal{G})$, $1 \leq p\leq \infty$, and the
mixed-norm spaces $L^{p,q}(\mathcal{G}$ provide some  natural
examples of solid Banach spaces on $\mathcal{G}$.
% For more sophisticated examples of solid Banach spaces we refer the reader to Section 4.
If $w$ is some positive measurable weight function on $\mathcal{G}$, then we  define
$L^p_w$  to be the set of all measurable function $F$ such that  $ Fw \in L^p$ with
$\|F|L^p_w\| := \|Fw|L^p\|$. A continuous weight $w$ is called submultiplicative
if $w(xy) \leq w(x) w(y)$ for all $x,y \in \mathcal{G}$. A weight
function $m$ is called  $w$-moderate
if $m(xyz) \leq C w(x) m(y) w(z)$, $x,y,z \in \mathcal{G}$.
It follows that   $L^p_m$ is  invariant under  left and right
translations,  if and only if $m$ is $w$-moderate.

As a next ingredient, we need certain discrete sets in $\mathcal{G}$.
Let $X=(x_i)_{i\in I}$ be some discrete set of points in $\mathcal{G}$ and
$U$ a relatively compact neighborhood of $e$ in $\mathcal{G}$.
\begin{itemize}\itemsep=-1pt
\item[(a)] $X$ is called $U$-{\em dense} if $\mathcal{G} = \bigcup_{i \in I} x_i U$.
\item[(b)] $X$ is called {\em relatively
separated} if for all compact sets
$K \subset \mathcal{G}$ there exists a constant $C_K$ such that
$\sup_{j \in I} \#\{ i\in I,\, x_iK \cap x_jK \neq \emptyset \} \leq C_K$.
\item[(c)] $X$ is called  well-spread if it is
both relatively separated and $U$-dense for
some $U$.
\end{itemize}

\begin{definition}\upshape
Given a well-spread family $X=(x_i)_{i\in I}$, and a 
relatively compact neighborhood $U$ of $e \in \mathcal{G}$, we define
the sequence space $Y_d$ associated to a solid Banach function space
$Y$ to be 
\begin{align}\label{def_Ydiscrete}
Y_d \,:=\, Y_d(X) \,:=\, Y_d(X,U) \,:=&\, \{ (c_i)_{i \in I}:\quad \sum_{i\in I}
c_i\chi_{x_i U} \in Y\},
\end{align}
endowed with the norm
$\|(c_i)_{i \in I} | Y_d\| := \|\sum_{i \in I} c_i \chi_{x_i U}|Y\|.$
\end{definition}
For instance, if $Y= L^p_w(\mathcal{G})$, then
$Y_d=L^p_w(\mathcal{G})_d = \ell^p_{\widetilde{w}}$,
where $\widetilde{w}$ is determined by  $\widetilde{w}_i=w(x_i)$.

If $L^\infty _0(\mathcal{G}$ is dense in $Y$, then the finite
sequences are dense in $Y_d$~\cite[Lemma 3.5(a)]{FeGr89_1}. 

%%%%%%%%%%%%%%%%%%%%%%%%%%%%%%%%%%%%%%%%%%%%%%%%%%%%%%%%%%%%%%%%%%%%Susbection
%%%%%%%%%%%%%%%%%%%%%%%%%%%%%%%%%%%%%%%%%%%%%%%%%%%%%%%%%%%%%%%%%%%%%%%%%%%%%%
\subsection{Wiener Amalgam Spaces}
%%%%%%%%%%%%%%%%%%%%%%%%%%%%%%%%%%%%%%%%%%%%%%%%%%%%%%%%%%%%%%%%%%%%%%%%%%%%%%
%%%%%%%%%%%%%%%%%%%%%%%%%%%%%%%%%%%%%%%%%%%%%%%%%%%%%%%%%%%%%%%%%%%%%%%%%%%%%%
Let $U$ be  some relatively compact neighborhood  of $e \in \mathcal{G}$.
We define the local maximum function of $F$  by
\begin{equation}\label{def_control}
F_\sharp(x) \,: =\sup_{y\in xU} |F(y)|, \quad x\in \mathcal{G},
\end{equation}
whenever  $F$ is locally bounded, in symbols $F\in L^{\infty}_{\mathrm{loc}}$.
Given a Banach space $Y$ of functions
on $\mathcal{G}$ satisfying \ref{subsection_Banach_function_spaces} (i)--(iii),
the  Wiener amalgam space $W(L^{\infty},Y)$ is  defined by
\[\label{def_Wiener_space}
W(L^{\infty},Y) \,:=\, \{F \in L^{\infty}_{\mathrm{loc}}:\quad F_\sharp \in Y\}
\]
equipped with the norm
\begin{equation}\label{qnormW}
\|F|W(L^{\infty},Y)\|\,:=\, \|F_\sharp|Y\|.
\end{equation}
Similarly, the right local maximum function is 
$F_\sharp ^R (x) =\sup_{y\in U\inv x\inv } |F(y)|$ and the right Wiener amalgam
space $W^R(L^{\infty},Y)$ is defined by the norm
$\|F|W^R(L^{\infty},Y)\|\,:=\, \|F_\sharp ^R|Y\|.$ 
By $W^R(C,Y)$ we denote the closed subspace of $W^R(L^\infty,Y)$ consisting of
continuous functions.
% We also need  the right Wiener Amalgam spaces $W^R(L^{\infty},Y)$.
% They are obtained by using the right local maximal function
% $F_\sharp^R(x) =  \|(R_x \chi_U) F|L^{\infty}\|$. 
% The Wiener amalgam spaces modeled on $M(\mathcal{G})$, the space of
% complex Radon measures, are defined accordingly. 
In several arguments we  need the following  convolution relation from~\cite[Proposition 5.2]{FeGr89_1}.
% \begin{proposition}\label{proposition_measures_convolution}
% \begin{itemize}
%   \item[(a)] Setting $v^*(x)=v(x^{-1})\Delta^{-1}(x)=|\!|\!|R_{x^{-1}}|Y|\!|\!|$, we have
% \[W(M,Y)*W^R(C,L^1_{v*})\subset Y.\]
% In other words, there is a constant $C>0$ such that
% \[\|\mu*F|Y\|\leq C \|\mu\;|W(M,Y)\| \;\|F\;|W^R(C,L^1_{v^*})\|\]
% for all $\mu \in W(M,Y)$ and $F\in W^R(C,L^1_{v*})$
%   \item[(b)] The discrete measure $\sum_{i\in I} c_i\delta_{x_i}$ belongs to $W(M,Y)$ if and only if
%   $(c_i)_{i\in I}$ belongs  to $Y_d$ for every relatively separated family $X=(x_i)_{i\in I}$
%   and the corresponding norms are equivalent.
% \end{itemize}
% \end{proposition}
% Combining statements (a) and (b), we obtain  immediately the following.
\begin{proposition}\label{corrolary_convolution_relation} If
  $(c_i)_{i\in I}\in Y_d$ and $H\in W^R(L^\infty, L^1_{w})$, then
  $\sum_{i\in I} c_i L_{x_i}H \in Y$ and 
\begin{align}\label{convolution_relation}
\Big\| \sum_{i\in I} c_i L_{x_i}H \;\Big| Y \Big\|\leq \;C\|(c_i)_{i\in I}\;|Y_d\| \|H\; |\; W^R(L^\infty, L^1_{w})\|.
\end{align}
The sum $\sum_{i\in I} c_i L_{x_i}H$  converges unconditionally in $Y$, 
  if $L^\infty _0$ is dense in $Y$, and otherwise $w^*$ in the
  $\sigma (Y, L^1_w)$-topology.  
\end{proposition}

%%%%%%%%%%%%%%%%%%%%%%%%%%%%%%%%%%%%%%%%%%%%%%%%%%%%%%%%%%%%%%Subsection
%%%%%%%%%%%%%%%%%%%%%%%%%%%%%%%%%%%%%%%%%%%%%%%%%%%%%%%%%%%%%%%%%%%%%%%%
\subsection{Coorbit Spaces}
%%%%%%%%%%%%%%%%%%%%%%%%%%%%%%%%%%%%%%%%%%%%%%%%%%%%%%%%%%%%%%%%%%%%%%%%
%%%%%%%%%%%%%%%%%%%%%%%%%%%%%%%%%%%%%%%%%%%%%%%%%%%%%%%%%%%%%%%%%%%%%%%%
Let $\pi$ be an irreducible unitary representation of $\mathcal{G}$
on a Hilbert space $\mathcal{H}$. For a fixed $g \in \mathcal{H}$, the abstract wavelet transform is defined as
\[
V_g f(x) \,:=\, \langle f, \pi(x) g\rangle, \quad f \in \mathcal{H}, \quad x\in\mathcal{G}.
\]
The representation $\pi$ is called square-integrable,  if there is  a
non-zero vector 
$g \in \mathcal{H}$, a so-called admissible vector, such that $V_g g \in L^2(\mathcal{G})$.
The main ingredient in coorbit space theory is  a reproducing formula of the form
\begin{equation}\label{rep_form_H}
V_g f \,=\, V_g f * V_g g, \qquad \text{ for all } f\in \mathcal{H}\, ,
\end{equation}
where $*$ denotes the  convolution on $\mathcal{G}$. Reproducing formulae are known to hold for many types of representations.
In particular, \eqref{rep_form_H} holds for every square-integrable,
irreducible representation $\pi$ of $\mathcal{G}$~\cite{grossmann-morlet85}
and also for many reducible square-integrable representations, see \cite{Fuhr_Book, GKT92}.
In order to introduce the coorbit spaces we first need to extend
the definition of the abstract wavelet  transform to a suitable space of distributions.
We define the following class
of analyzing vectors
$$
\mathbb{A}_w \,:=\, \{g \in \mathcal{H}:\quad  V_g g \in L^1_w\}.
$$
Let us assume that $\mathbb{A}_w$ is non-trivial, i.e., $\pi$ is
integrable, then 
 $\pi$ is also square-integrable. For a fixed $g \in \mathbb{A}_w \setminus \{0\}$
we define
\[
\mathcal{H}_w^1 \,:=\, \{f \in \mathcal{H}:\quad V_g f \in L^1_w\}
\]
endowed with the norm $\|f|\mathcal{H}_w^1\|:= \|V_g f|L^1_w\|$.
Further, we denote by $(\mathcal{H}_w^1)^\urcorner$ the anti-dual, i.e.,
the space of all bounded conjugate-linear functionals on $\mathcal{H}^1_w$.
An equivalent norm
on $(\mathcal{H}^1_w)^\urcorner$ is given by $\|V_g f|L^\infty_{1/w}\|$.
Since the inner product on $\mathcal{H}\times \mathcal{H}$ extends to a sesquilinear form on
$(\mathcal{H}^1_w)^\urcorner\times \mathcal{H}^1_w$, the extended representation coefficients
\[
V_g f(x) \,=\, \langle f,\pi(x) g\rangle, \quad f \in (\mathcal{H}^1_w)^\urcorner,\quad  g\in \mathbb{A}_w
\]
are well-defined.
We are now in  a position to define  coorbit spaces.
\begin{definition}\label{def_coorbit}\upshape
 Let $Y$ be a solid Banach space of functions on $\mathcal{G}$ with canonical weight $w$.
Then for $g\in \mathbb{A}_w, g \neq 0,$ the coorbit space is defined by
\[
\mathrm{Co} Y \,:=\, \{f \in (\mathcal{H}^1_w)^\urcorner:\quad V_g f \in Y \}
\]
with the norm $\|f|\mathrm{Co}Y\| := \|V_g f|\;Y\|$.
\end{definition}
\begin{remark}\upshape
$\mathcal{H}_w^1$, $(\mathcal{H}^1_w)^\urcorner$,  and $\mathrm{Co} Y$ are  $\pi$-invariant Banach spaces.
If $\pi$ is irreducible, then their
definitions do not depend on the choice of the analyzing vector $g$ in
the sense that different windows provide equivalent
norms~\cite[Thm.~4.2]{FeGr89_1}. 
\end{remark}

%%%%%%%%%%%%%%%%%%%%%%%%%%%%%%%%%%%%%%%%%%%%%%%%%%%%%%%%%%%%%%Subsection
%%%%%%%%%%%%%%%%%%%%%%%%%%%%%%%%%%%%%%%%%%%%%%%%%%%%%%%%%%%%%%%%%%%%%%%%
\subsection{Atomic Decomposition and Banach Frames}
%%%%%%%%%%%%%%%%%%%%%%%%%%%%%%%%%%%%%%%%%%%%%%%%%%%%%%%%%%%%%%%%%%%%%%%%
%%%%%%%%%%%%%%%%%%%%%%%%%%%%%%%%%%%%%%%%%%%%%%%%%%%%%%%%%%%%%%%%%%%%%%%%
Next  we describe atomic decompositions and Banach frames in
coorbit spaces as outlined in \cite{FeGr89_1,FeGr89_2,Gr91}.
The treatment of  coherent frames for $\mathrm{Co}Y$ requires a
further  restriction of  the set of analyzing vectors.
The set of "better vectors" is given by
$$
\mathbb{B}_w \,:=\, \{g \in \mathcal{H}:\quad  V_g g \in W^R(L^{\infty},L^1_w)\}.
$$
If $\mathcal{A}_w \neq \emptyset$, then also $\mathcal{B}_w \neq \emptyset$.

%Further, we  need the following maximal function.
%For some relatively compact neighborhood $U$ of $e \in \mathcal{G}$ and a function
%$G$ on $\mathcal{G}$ the $U$-oscillation is given by
%\[
%G_U^\#(x) \,:=\, \sup_{u \in U} |G(ux) - G(x)|.
%\]

Below we summarize the  results about the existence of atomic
decompositions and frames from \cite[Theorem U]{Gr91}.

\begin{theorem}\label{atomic_decomposition}
Let $Y$ satisfy \ref{subsection_Banach_function_spaces} (i)-(iii) with canonical weight $w$ given by \eqref{weight_w}
and  assume that $g\in \mathbb{B}_w, g \neq 0$.  Then there exists a
neighborhood $U$ of $e$  such that
%\[\|(V_g g)^{\sharp}_U|L^1_w\|\left(\|V_g g|L^1_w\|+c\|V_g g|W^R(C_0,L^1_w)\|\right)<1.\]
for any $U$-dense and relatively separated family $X=(x_i)_{i\in I}$ in $\mathcal{G}$ the set $\{\pi(x_i)g\}_{i\in I}$
provides an atomic decomposition and a Banach frame for $\mathrm{Co} Y$.
\begin{description}
\item[\bf A. (Atomic decomposition)] Every $f\in \mathrm{Co}Y$
  possesses an expansion 
  \begin{equation}
    \label{eq:fi1}
      f= \sum_{i\in I} c_i(f)\pi(x_i)g,
  \end{equation}
where the sequence of coefficients $(c_i(f))_{i\in I}$ depends linearly on $f$ and satisfies
\[\|(c_i(f))_{i\in I}|Y_d\|\leq C \|f|\mathrm{Co}Y\|\]
with a constant $C$ depending only of $g$.\\
Conversely, if $(c_i)_{i\in I}\in Y_d$, then $f= \sum_{i\in I} c_i\pi(x_i)g$ is in $\mathrm{Co}Y$ and
 \[ \|f|\mathrm{Co}Y\| \leq C'\|(c_i)_{i\in I}|Y_d\|.\]
The series defining $f$  converges unconditionally in the  normof
$\mathrm{Co}Y$, if $L_0^{\infty}(\mathcal{G})$ is dense in $Y$,
otherwise it converges unconditionally  in the weak--* topology of
           $(\mathcal{H}^1_w)^\urcorner$.

\item[\bf B. (Banach frames)] $\{\pi(x_i)g\}_{i\in I}$ is a Banach
  frame for $\mathrm{Co}Y$. This means that 
\begin{itemize}
\item[(i)] There are two constants $C_1,C_2>0$ depending only on $g$ such that
\[C_1 \|f|\mathrm{Co}Y\| \leq \|(\langle f,\pi(x_i)g \rangle)_{i\in
  I}|Y_d\| \leq C_2 \|f|\mathrm{Co}Y\|.\] 
\item[(ii)] (Reconstruction operator) There exists a bounded mapping $R$ from $Y_d(X)$ onto
  $\mathrm{Co} Y$, such that $ f= R(\langle f,\pi(x_i)g \rangle _{i\in
    I})$.  % can be  reconstructed from  the
%   coefficients  $\langle f,\pi(x_i)g \rangle$, ${i\in I}$. For
%   instance, 
%  there exists a system $e_i\in \mathcal{H}^1_w$ , ${i\in I}$ such that
%  \[f= \sum_{i\in I}\langle f,\pi(x_i)g \rangle e_i \]
%  with convergence in $\mathrm{Co}Y$, provided that
% $L_0^{\infty}(\mathcal{G})$ is dense in $Y$.

\end{itemize}
\item[\bf C. (Dual frames)] There exists a "dual frame" $\{e_i\}_{i\in
    I}$ in $\mathcal{H}^1_w$, such that,  for every  $f\in
  \mathrm{Co}Y$, 
           \[f= \sum_{i\in I}\langle f,e_i \rangle \pi(x_i)g \, ,\]
and $\|(\langle f,e_i\rangle)_{i\in I}|Y_d\| $ is an equivalent norm
on $\mathrm{Co} Y$. 
%  with the following properties: 
% \begin{itemize}
% \item[(i)] The following norms are equivalent.
%            \[\|f|\mathrm{Co}Y\|\sim \|(\langle f,e_i\rangle)_{i\in I}|Y_d\|\sim \|(\langle f, \pi(x_i)g\rangle)_{i\in I}|Y_d\|.\]
% \item[(ii)] 
%\item[(iii)] If $L_0^{\infty}(\mathcal{G})$ is dense in $Y$, then the decomposition
%            \[f= \sum_{i\in I}\langle f,\pi(x_i)g \rangle e_i \]
%            is also valid for  $f\in \mathrm{Co}Y$.
\end{description}
\end{theorem}

%%%%%%%%%%%%%%%%%%%%%%%%%%%%%%%%%%%%%%%%%%%%%%%%%%%%%%%%%%%%Section
%%%%%%%%%%%%%%%%%%%%%%%%%%%%%%%%%%%%%%%%%%%%%%%%%%%%%%%%%%%%%%%%%%%
\section{Molecules in Coorbit Space Theory}
%%%%%%%%%%%%%%%%%%%%%%%%%%%%%%%%%%%%%%%%%%%%%%%%%%%%%%%%%%%%%%%%%%%%
%%%%%%%%%%%%%%%%%%%%%%%%%%%%%%%%%%%%%%%%%%%%%%%%%%%%%%%%%%%%%%%%%%%
In this section we introduce the notion of molecules in coorbit spaces
and state and prove our  main result. 
\begin{definition} \label{definition_molecules}\upshape
Assume that  $g \in \mathbb{B}_w, g\neq 0,$ and let  $X=(x_i)_{i\in I}$ be a well-spread family in $\mathcal{G}$.
A collection of functions $\{m_i\}_{i\in I}\subset \mathcal{H}$ is called a {\em set of molecules},
if there exists an envelope function $H\in  W^R(L^\infty, L^1_w)$ such that
\begin{align}\label{formula_definition_molecules}
|V_g m_i(z)|\leq L_{x_i}H(z), \quad i\in I.
\end{align}
\end{definition}

\begin{remark}\upshape
We may think of $\mathcal{G}$ as a kind of phase space and the function $V_g f$  (for fixed $g\neq 0$) as a phase-space representation
of $f$. The molecule $m_i$ is then localized at $x_i\in \mathcal{G}$ and a set of molecules has a uniform envelope in phase-space.
In other words, each molecule possesses the same phase-space concentration.
\end{remark}

\begin{example}\upshape
1. Every set of atoms $\{\pi(x_i)g\}_{i\in I}$ for $g\in
    \mathcal{B}_w$  is a set of molecules in the sense of
    Definition~\ref{definition_molecules}, because $|\langle \pi (x_i)
    g, \pi (z) g\rangle | = |\langle     g, \pi (x_i ^{-1}z) g\rangle
    | = L_{x_i} |V_gg(z)|$ and $V_gg \in  W^R(L^\infty, L^1_w)$.

2. Fix $g_0 \in \mathbb{B}_w$ and a positive  function $H\in  W^R(L^\infty, L^1_w)$, and set
\[C_\mathcal{H}:=\{g \in \mathcal{H}:\quad |V_{g_0} g(x)| \leq H(x) \}.\]
If $X=(x_i)_{i\in I}$ is well-spread and $g_i\in C_\mathcal{H}$,
then %by intertwining relation (cf. formula (2.4) in \cite{FeGr89_1})
\[|V_{g_0} (\pi (x_i)g_i)(z)|= |\langle \pi (x_i)g_i, \pi (z)
g_0\rangle| =  |L_{x_i}V_{g_0}g(z) |\leq L_{x_i}H(z) \]
and so the set
$\{m_i=\pi(x_i)g_i\}_{i\in I}$ forms a family of $H$-molecules.
\end{example}

In preparation for the main result, we verify  the following  basic
properties of molecules.
%%%%%%%%%%%%%%%%%%%%%%%%%%%%%%%%%%%%%%%%%%%%%%%%%%%%%%%%%%%%%Synstesis Lemma%%%%%%%%%%%%%%%%%%%%%%%%%%%%%%%%
%%%%%%%%%%%%%%%%%%%%%%%%%%%%%%%%%%%%%%%%%%%%%%%%%%%%%%%%%%%%%%%%%%%%%%%%%%%%%%%%%%%%%%%%%%%%%%%%%%%%%%%%%%%%
\begin{lemma}\label{synthesis_lemma}

  (i) The definition of molecules does not depend on the particular
  choice of the window $g\in \mathbb{B}_w$. 

(ii) \textbf{Synthesis.} Let $\{m_i\}_{i\in I}$ be a set of
    molecules subordinated to $H\in W^R (L^\infty, L^1_{w})$. 
 The synthesis operator $(c_i)_{i\in I} \to \sum _{i\in I} c_i m_i$ is
 bounded from $Y_d$ to $  \mathrm{Co}Y$. If  $(c_i)_{i\in I}\in Y_d$,
 then $f=\sum_{i\in I} c_i m_i \in  \mathrm{Co} Y$  and 
   \begin{align}\label{synthesis_lemma_formula_1}
    \Big\| \sum_{i\in I} c_i m_i\;|\mathrm{Co}Y \Big\|\leq
    C\|(c_i)_{i\in I}\;|Y_d\| \|H|\;W^R(L^\infty, L^1_{w})\|\, , 
   \end{align}
for some constant $C$. The sum defining $f$ converges unconditionally,
whenever $L^\infty _0(\mathcal{G})$ is dense in $Y$, and in the
$w^*$-sense on $(\mathcal{H}^1_w)^\urcorner $ otherwise. 

  (iii) \textbf{Analysis.} If, in addition,  $H\in  W(L^\infty, L^1_w)$, then the coefficient operator
   $Cf:=(\langle f,m_i\rangle)_{i\in I}$ is bounded from  $\mathrm{Co}Y$ to $Y_d$ with
   \[\|(\langle f,m_i\rangle)_{i\in I}\;|Y_d\| \leq  C \|f\;| \mathrm{Co}Y\|. \]
\end{lemma}
%%%%%%%%%%%%%%%%%%%%%%%%%%%%%%%%%%%%%%%%%%%%%%%%%%%%%%%%%%%%%Synstesis Lemma%%%%%%%%%%%%%%%%%%%%%%%%%%%%%%%%
%%%%%%%%%%%%%%%%%%%%%%%%%%%%%%%%%%%%%%%%%%%%%%%%%%%%%%%%%%%%%%%%%%%%%%%%%%%%%%%%%%%%%%%%%%%%%%%%%%%%%%%%%%%%
\begin{proof}
To prove {\em (i)} we assume that $g,h \in  \mathbb{B}_w$ and that
$\{\pi (z_j)g: j\in J\} $ is a (Banach) frame for $\mathcal{H}^1_w$.
After plugging  the frame expansion of $h=\sum_{j\in J} \langle h,e_j
\rangle \pi(z_j)g$, 
where the sequence $(c_j)_{j\in I}$ with $c_j:=|\langle h,e_j \rangle|$ is in $\ell^1_{\widetilde{w}}$,
 into \eqref{formula_definition_molecules} we obtain that 
\begin{align*}\label{}
|V_h m_i(z)|&=|\langle m_i,\pi(z)h \rangle)|\leq \sum_{j\in I}|\langle h,e_j \rangle||\langle  m_i, \pi(z)\pi(z_j)g\rangle|\\
&\leq  \sum_{j\in I} c_j L_{x_i}H(zz_j)=L_{x_i}\Big(\sum_{j\in I}c_j R_{z_j}H(z)\Big).
\end{align*}
Since $W^R(L^\infty,L^1_w)$ is invariant under right translations, we
find that  $\widetilde{H}= \sum_{j\in I}c_j R_{z_j}H\in
W^R(L^\infty,L^1_w)$, and
\eqref{formula_definition_molecules} is  satisfied for $h$ in place of
$g$ with the envelope function $\widetilde{H}$.

For the proof of {\em (ii)} we assume first  that
$L^\infty_0(\mathcal{G})$ is dense in $Y$. In this case, the finite
sequences are dense in $Y_d$ by ~\cite[Lemma~3.5]{FeGr89_1} and thus  it suffices 
to prove \eqref{synthesis_lemma_formula_1} for finite sequences.
If supp$(c)$  is finite, then  by the solidity of $Y$ and the property of molecules we obtain
\begin{align*}
\Big\| \sum_{i\in I} c_i m_i |\mathrm{Co}Y\Big\|&=\Big\|V_g\left ( \sum_{i\in I} c_i m_i\right) \;\Big|Y \Big\|
\leq \Big\| \sum_{i\in I} |c_i| |V_g m_i|\; \Big|Y \Big\|
\label{norm12} \\
&\leq \Big\| \sum_{i\in I} |c_i| L_{x_i}H \;\Big| Y \Big\|\leq \;C\|(c_i)_{i\in I}\;|Y_d\| \|H\; |\; W^R(L^\infty, L^1_{w})\|.
\end{align*}
The last inequality above follows immediately from
Proposition~\ref{corrolary_convolution_relation}. This norm estimate
also implies the unconditional convergence in $\mathrm{Co} Y$.

If $L^\infty _0(\mathcal{G})$ is not dense in $Y$, then still $\sum
_{i\in I} |c_i| L_{x_i} H \in Y$, but the sum converges only in the
weak-$^*$ sense. Thus the above  estimate  still holds, and 
$\sum_{i\in I} c_i m_i \in \mathrm{Co}Y$ is $w^*$-convergent.

Finally, we show {\em (iii)}. By virtue of Theorem
\ref{atomic_decomposition}, every $f\in \mathrm{Co}Y$ possesses an
expansion 
\[f= \sum_{j\in J}c_j \pi(z_j)g\]
with $(c_j)_{j\in J}\in Y_d$ and $\|(c_j)_{j\in J}|Y_d\| \leq C \|f |
\mathrm{Co} Y\|$. Plugging again the above expansion yields
\begin{align*}\label{}
|\langle f, m_i \rangle)|&\leq \sum_{j\in J}|c_j|\, |\langle \pi(z_j)g, m_i\rangle|\\
&\leq  \sum_{j\in J} c_j H(x_i^{-1}z_j)=\sum_{j\in J}c_j L_{z_j}H^\vee(x_i).
\end{align*}
Consequently, we get
\begin{align*}\label{}
\|(\langle f,m_i\rangle)_{i\in I}\;|Y_d\| &\leq \Big\|\sum_{j\in J}c_j L_{z_j}H^\vee | W(L^{\infty}, Y)\Big\|\\
&\leq \|(c_i)_{i\in I}\;|Y_d\| \|H^{\vee}\; |\; W^R(L^\infty, L^1_{w})\| \leq C \|f\;| \mathrm{Co}Y\|.
\end{align*}
\end{proof}
Now we formulate our  main result on the boundedness of operators on
coorbit spaces. 
\begin{theorem}\label{main_theorem}
Suppose that $g\in \mathbb{B}_w$ and that  $\{\pi(x_i)g\}_{i\in I}$
forms a Banach  frame for $\mathrm{Co} Y$
with canonical dual frame $\{e_i\}_{i\in I}$ (as guaranteed by
Theorem~\ref{atomic_decomposition}).

Assume that the operator $T$ is  bounded from   $\mathcal{H}^1_w$ to
$(\mathcal{H}^1_w)^\urcorner$ and that $T$ 
maps the atoms $\pi(x_i)g, i\in I,$ to the set of molecules  $m_i=
T(\pi(x_i)g)$ with envelope $H\in W^R(L^\infty, L^1_w)$.  
Then $T$ extends to a  bounded operator on
$\mathrm{Co}Y$. Furthermore, the operator norm of  $T$ is bounded by 
$\|H\; |\; W^R(L^\infty, L^1_{w})\|$.
\end{theorem}
\begin{proof}
For  $f= \sum _{i\in I} c_i \pi (x_i)g$, we would like to define $Tf =
\sum _{i\in I} c_i T(\pi (x_i)g) = \sum _{i\in I} c_i
m_i$. Lemma~\ref{synthesis_lemma} then yields the correct norm
estimates. However, in general, the representation of $f$ with respect
to $\{\pi (x_i)g\}$ is not unique, therefore we have to show that the
natural extension procedure is unique.

{\tt Step 1.} First we  define a canonical extension $\widetilde{T}$
of $T$ to $\mathrm{Co}Y$ via the frame expansion of $f$. Let $e_i \in
\mathcal{H}_w^1$ be the dual frame of $\pi (x_i)g, i \in I$, the
existence of which is asserted in
Theorem~\ref{atomic_decomposition}(C). Then $f\in \mathrm{Co}Y $ has
the expansion 
$  f= \sum_{i\in I}\langle f,e_i \rangle \pi(x_i)g $ with coefficient
sequence $(\langle f,e_i\rangle )_{i\in I} \in Y_d$ and 
\begin{align}\label{}
\|(\langle f,e_i \rangle)_{i\in I}\;|Y_d\|\leq C \|f\;| \mathrm{Co}Y\|,
\end{align}
where the constant $C>0$ is independent of $f$.
We define $\widetilde{T}f$ by 
 \begin{align}\label{proof_step_2_2}
 \widetilde{T}f= \sum_{i\in I}\langle  f, e_i\rangle T(\pi(x_i)g)= \sum_{i\in I}\langle  f, e_i\rangle m_i.
 \end{align}
By Lemma~\ref{synthesis_lemma}(ii) we find that $\widetilde{T}f$ is in
$\mathrm{Co}Y$ and that 
\begin{eqnarray}
  \|\widetilde{T}f\;|\mathrm{Co}Y\| &\leq &   C\|(\langle f,e_i
\rangle)_{i\in I}\;|Y_d\| \|H\; |\; W^R(L^\infty, L^1_{w})\| \notag \\
&\leq & \;C'\|H\; |\; W^R(L^\infty, L^1_{w})\| \|f\;|\mathrm{Co}Y \|. 
  \label{eq:ch2}
\end{eqnarray}
% \begin{equation}
%   \label{eq:ch2}
% \|\widetilde{T}f\;|\mathrm{Co}Y\| \leq   C\|(\langle f,e_i
% \rangle)_{i\in I}\;|Y_d\| \|H\; |\; W^R(L^\infty, L^1_{w})\|\leq
% \;C'\|H\; |\; W^R(L^\infty, L^1_{w})\| \|f\;|\mathrm{Co}Y \|. 
%   \end{equation}
Furthermore, the series defining $\widetilde{T}f$ converges
unconditionally in $\mathrm{Co}Y$, if $L^\infty _0$ is dense in $Y$,
and $w^*$ in $(\mathcal{H}^1_w)^\urcorner$ otherwise.

{\tt Step 2:} It remains to be shown that $\widetilde{T}$ coincides
with $T$ on $\mathcal{H}_w^1$. Here 
we  exploit the  assumed continuity of  $T$ from   $\mathcal{H}^1_w$
to $(\mathcal{H}^1_w)^\urcorner$. 
This  means that the convergence $f_n\rightarrow f$ in
$\mathcal{H}^1_w$ implies the {\em w}*-convergence  $Tf_n\rightarrow Tf$. 
In particular, for $\pi (x_i)g$ the net  of partial sums
\[f_{F}=\sum_{k\in F}\langle \pi(x_i)g,e_k \rangle\pi(x_k)g \]
converges to $\pi(x_i)g$ as $F \rightarrow I$, where $(F)$ is the  net
of finite subsets of $I$ ordered by inclusion. 
Consequently,
\begin{eqnarray}
m_i &=& T(\pi (x_i)g) = w^*-\lim Tf_{F} \notag  \\
&=&w^*-\lim _{F\to I} \sum_{k\in F}\langle \pi(x_i)g, e_k\rangle T(\pi
(x_i)g) \notag \\
&=& w^*-\lim _{F\to I} \sum_{k\in F}  \langle \pi(x_i)g, e_k\rangle m_k
\label{eq:ll12} \\
&=& \widetilde{T} ((\pi (x_i)g)) \, . \notag
\end{eqnarray}
%with {\em w}*- convergence in $(\mathcal{H}^1_w)^\urcorner$.
Since the  $m _k$'s are molecules, the  series in \eqref{eq:ll12} converges also  in
$\mathcal{H}^1_w$ by Lemma~\ref{synthesis_lemma}. 
The identity $T(\pi (x_i)g ) = \widetilde{T}(\pi (x_i)g )$  implies that
$Tf= \widetilde{T}f$ whenever $f= \sum 
_{i\in I } c_i \pi (x_i) g$ and $(c_i )_{i\in I} \in \ell ^1_w$. 

%If the finite sequences are dense in $Y_d$, then we may take
We now take $\widetilde{T}$ as the desired   extension of $T$ from
$\mathcal{H}_w^1$ to $\mathrm{Co} Y$.  By Step~1 this extension is
bounded on $\mathrm{Co}Y$. 
This completes the proof.
\end{proof}

% \begin{remark}\upshape
% Note that by virtue of the above proof the canonical dual frame $\{e_i\}_{i\in I}$ described in Theorem \ref{atomic_decomposition} C
% provides an another example of a set of molecules.
% \end{remark}
\begin{remark}\upshape
  We observe that Theorem~\ref{main_theorem} asserts the simultaneous
  boundedness of $T$ on all coorbit spaces $\mathrm{Co} Y$ that
  possess the same associated weight $w$ given in \eqref{weight_w}. 
\end{remark}

%%%%%%%%%%%%%%%%%%%%%%%%%%%%%%%%%%%%%%%%%%%%%%%%%%%%Section
%%%%%%%%%%%%%%%%%%%%%%%%%%%%%%%%%%%%%%%%%%%%%%%%%%%%%%%%%%%
\section{Examples and Applications}
%%%%%%%%%%%%%%%%%%%%%%%%%%%%%%%%%%%%%%%%%%%%%%%%%%%%%%%%%%%
%%%%%%%%%%%%%%%%%%%%%%%%%%%%%%%%%%%%%%%%%%%%%%%%%%%%%%%%%%%
\subsection{The Heisenberg Group and Time-Frequency Molecules}  
We now describe the consequences of Theorem \ref{main_theorem}
in the context of time-frequency molecules. Time-frequency molecules were introduced in \cite[Section 5.3]{Gr04}
and independently in \cite{bacahela06-1} and were  studied  in detail in \cite[Section 7]{GrRz07}.

We first discuss how the modulation spaces fit into  coorbit space setting.
We consider  the $d$-dimensional reduced Heisenberg group $\mathcal{G}_{\mathrm{H}} = 
\mathbb{R}^d \times \mathbb{R}^d \times \mathbb{T}$ 
 with multiplication
\[(x,\omega,\tau) (x',\omega',\tau') =
(x+x',\omega + \omega', \tau \tau' e^{\pi i(x'\cdot \omega - x\cdot \omega')}).\]
Let $T_xf(t)=f(t-x)$ and $M_{\omega}f(t)=e^{2\pi it \cdot\omega}f(t)$ be the operators of translation and modulation, respectively, and
$\pi$ be  the Schr\"odinger representation of $\mathcal{G}_{\mathrm{H}}$ acting on $L^2(\mathbb{R}^d)$
by time-frequency shifts
\[
\pi(x,\omega,\tau) \,:=\, \tau e^{\pi i x\cdot \omega} T_x M_\omega
\,=\, \tau e^{-\pi i x \cdot \omega} M_\omega T_x.
\]
This is an irreducible unitary and square-integrable representation of $\mathcal{G}_{\mathrm{H}}$.
Except  for  a  trivial phase factor the representation coefficient $V_g f(x,\omega,\tau) = \langle f, \pi(x,\omega,\tau) g\rangle_{L^2(\mathbb{R}^d)}$
coincides with the  Short-Time Fourier Transform (STFT) given by
\begin{align}
\mathrm{STFT}_g f(x,\omega)=\langle f, M_\omega T_x g\rangle_{L^2(\mathbb{R}^d)}=
\int_{\mathbb{R}^d}
f(t)\overline{g(t-x)} e^{-2\pi i \omega \cdot t} \mathrm{d} t \, ,
\end{align}
whenever the integral makes sense. Otherwise, we fix $g\in
\mathcal{S}(\mathbb{R}^d)$ and extend the STFT to tempered
distributions $\mathcal{S}'(\mathbb{R}^d)$ by interpreting 
the bracket $\langle f,g  \rangle$  as a dual pairing between an
element  $f\in \mathcal{S}'(\mathbb{R}^d)$ and
$g\in \mathcal{S}(\mathbb{R}^d)$.  For more information on the STFT the
reader is referred to \cite{Gr01_Book}.

We take the liberty to drop the center  $\{0\}\times \{0\}\times
\mathbb{T}$ of   $\mathcal{G}_{\mathrm{H}}$ and consider function spaces on $\mathbb{R}^{2d}$
instead of $\mathcal{G}_{\mathrm{H}}$. As a standard example we take the mixed-norm spaces
$L^{p,q}_m(\mathbb{R}^{2d})$ for $1\leq p,q\leq \infty$ and some $w$-moderate weight function $m$ on $\mathbb{R}^{2d}$ with the norm
\[
\|F\;|L^{p,q}_m(\mathbb{R}^{2d})\| \,:=\, \left(\int_{\mathbb{R}^d}\left(\int_{\mathbb{R}^d}
|F(x,\omega)|^p m(x,\omega)^p \mathrm{d}x\right)^{q/p}\mathrm{d}\omega\right)^{1/q}.
\]
The modulation spaces are obtained as the coorbits of
$L^{p,q}_m(\mathbb{R}^{2d})$ with respect to the Schr\"odinger representation $\pi$

\[
M^{p,q}_m(\mathbb{R}^d) \,=\, \mathrm{Co} L^{p,q}_m(\mathbb{R}^{2d}) \,=\, \{f \in \mathcal{S}'(\mathbb{R}^{d}):\quad \mathrm{STFT}_g f \in L^{p,q}_m(\mathbb{R}^{2d})\}.
\]
for fixed non-zero $g\in \mathcal{S}(\mathbb{R}^d)$. % Compare
                                % \cite{Gr91} for further remarks. 
For the Heisenberg group many technical
subtleties of the general set-up of coorbit space theory  disappear.
For instance $W(L^\infty,
L^1_w)(\mathcal{G}_{\mathrm{H}})=W^R(L^\infty,
L^1_w)(\mathcal{G}_{\mathrm{H}})$ 
and 
$$\mathbb{B}_w=\mathbb{A}_w=M^{1,1}_w(\mathbb{R}^{2d})$$
 (cf. \cite[Lemma 7.2]{FeGr89_2}).
As long as $w$ and $m$ have polynomial growth, one may use the
Schwartz class $\mathcal{S}(\mathbb{R}^d) \subset \mathcal{A}_w$ as a
convenient  space of test functions. 

In the context of modulation spaces and the Heisenberg group, the
natural discrete sets are lattices, i.e., discrete co-compact
subgroups of the form $\Lambda = A \mathbb{Z}^{2d}$ for some
invertible $2d \times 2d$-matrix $A$.  Let $G(g, \Lambda ):=\{\pi
(\lambda)g:\; \lambda \in \Lambda \}$ be the orbit of $g$ under
$\Lambda $ (a so-called \emph{Gabor system}).

Given a symbol $\sigma \in \mathcal{S}'(\mathbb{R}^{2d})$, the
pseudodifferential operator $\sigma ^w$ is informally given by 
\[\sigma ^w f= \int_{\mathbb{R}^d} \int_{\mathbb{R}^d}\widehat{\sigma}(\xi, u)
e^{-\pi i \xi u} T_{-u}M_\xi f \;\mathrm{d}u \; \mathrm{d}\xi,\]
whenever the integral makes sense, otherwise it is interpreted in the weak sense. 	
The mapping $\sigma \mapsto \sigma ^w$ is called the Weyl transform.

The abstract  Definition  \ref{definition_molecules} can be rephrased
as follows (cf. \cite{Gr04} and 
\cite[Definition 7.1]{GrRz07}). 
\begin{definition}\upshape \label{TF_molecules}
% Assume that $G(g, \Lambda )$ is a Gabor frame for
% $L^2(\mathbb{R}^d)$ and that
Fix a non-zero $g\in M^{1,1}_w(\mathbb{R}^d)$. 
A collection of functions $\{m_{\lambda}\}_{\lambda \in \Lambda }$
forms a {\em set of time-frequency molecules}, 
if there exists a function $H\in  L^1_w(\mathcal{G}_{\mathrm{H}})$ such that
\begin{align*}
|\langle m_\lambda,\pi(z)g \rangle|\leq H(z-\lambda), \quad \lambda
\in \Lambda \, .
\end{align*}
\end{definition}
In our language, the main theorem of \cite{Gr06} (cf. also
\cite[Proposition 7.1]{GrRz07}) can be formulated as follows. [We
write $j$ for the rotation mapping $j(z_1,z_2)=(z_2,-z_1)$ with
$(z_1,z_2)\in \mathbb{R}^{2d}$.] 

\begin{proposition}
Fix a non-zero $g\in M_w^{1,1}$  and suppose that $\mathcal{G}(g, \Lambda)$ is a Gabor frame for $L^2(\mathbb{R}^d)$.
Then the following are equivalent.

\begin{itemize}
\item[(i)] $\sigma\in M_{w\circ j^{-1}}^{\infty,1}(\mathbb{R}^{2d})$.

\item[(ii)] There exists a function $H\in
  L^1_w (\mathbb{R}^{2d})$ such that 
$$
|\langle \sigma ^w \pi (w)g , \pi (z) g\rangle | \leq H(z-w), \qquad
w,z \in \mathbb{R}^{2d} \, .
$$
\item[(iii)] There is a function $H\in W(L^\infty , L^1_w)(\mathbb{R}^{2d})$ such
  that the corresponding pseudodifferential operator $\sigma ^w$ 
            maps the  time-frequency shifts $\{\pi(\lambda )g\}$ to
time-frequency  molecules $\{m_\lambda\}_{\lambda \in \Lambda }$ in
the sense of     Definition~\ref{TF_molecules} 
           with envelope function $H$.
\end{itemize}
\end{proposition}

\begin{proof}
 The equivalence  $(i) \Leftrightarrow (ii)$ was proved in
 \cite{Gr04}. % , where it was also shown that the decay function can be
%  chosen in $W(L^\infty , L^1_w)$. 

Let $m_\lambda = \sigma ^w(\pi (\lambda )g)$, then by (ii) we have 
$|\langle m_\lambda , \pi (z) g\rangle | \leq H(z-\lambda )$, and thus the
set  of $m_\lambda  , \lambda \in \Lambda $, is a set of
time-frequency molecules in the sense of
Definition~\ref{TF_molecules}.

Conversely, if $m_\lambda = \sigma ^w(\pi (\lambda )g)$ is a set of
molecules, then we have $|\langle \sigma ^w(\pi (\lambda )g) , \pi
(\mu) g\rangle | \leq H(\mu -\lambda  )$. Again, by  \cite{Gr06} this
property implies that $\sigma \in M_{w\circ
  j^{-1}}^{\infty,1}(\mathbb{R}^{2d})$. 
\end{proof}

Since the modulation spaces are the coorbit spaces for the
Schr\"odinger representation, Theorem~\ref{main_theorem} now implies
the boundedness of pseudodifferential 
operators with symbol in $M_{w\circ
  j^{-1}}^{\infty,1}(\mathbb{R}^{2d})$ on a large class of modulation
spaces. See ~\cite[Thm.~14.5.6]{Gr01_Book} and~\cite{Toft04a} for different
proofs. 

\begin{corollary}
If $\sigma\in M_{w\circ j^{-1}}^{\infty,1}(\mathbb{R}^{2d})$, then
$\sigma ^w$ is bounded simultaneously on all modulation spaces
$M^{p,q}_m(\mathbb{R}^{d})$ for $1\leq p,q\leq \infty $ and every
$w$-moderate weight function $m$. 
\end{corollary}

%%%%%%%%%%%%%%%%%%%%%%%%%%%%%%%%%%%%%%%%%%%%%%%%%%%%Subsection Affine Group
%%%%%%%%%%%%%%%%%%%%%%%%%%%%%%%%%%%%%%%%%%%%%%%%%%%%%%%%%%%%%%%%%%%%%%%%%%%
\subsection{The Affine Group and Time-Scale Molecules}
%%%%%%%%%%%%%%%%%%%%%%%%%%%%%%%%%%%%%%%%%%%%%%%%%%%%%%%%%%%%%%%%%%%%%%%%%%
%%%%%%%%%%%%%%%%%%%%%%%%%%%%%%%%%%%%%%%%%%%%%%%%%%%%%%%%%%%%%%%%%%%%%%%%%%

We next consider the affine group
$\mathcal{G}_{\mathrm{A}} = \mathbb{R}^d \times \mathbb{R}_+$
%Here $\mathbb{R}_+^*$  denotes the multiplicative group of positive real numbers.
 with multiplication $(x,s)\cdot(x',s') = (x+sx',ss')$ for $x,x'\in \mathbb{R}^d$ and $s,s'>0$.
Let the dilation operator  be  given by $D_sf(x)=s^{-d/2}f(s^{-1}x)$ with  $s>0$.
A unitary representation of  $\mathcal{G}_{\mathrm{A}}$ acts on $L^2(\mathbb{R}^d)$  by translations and dilations:
\[\pi(x,s) g(t)=T_xD_s g(t)= s^{-d/2}g\left(\frac{t-x}{s}\right).\]

This representation  is square-integrable but reducible. Nevertheless it possesses an abundance of admissible vectors $g$ for which
the reproducing formula \eqref{rep_form_H} holds. Another way to deal with the reducibility is to study the extended affine group
$\mathbb{R}^d \times \mathbb{R}_+ \times SO(d)$ and its representations $\pi_1(x,s,R) f(t)=s^{-d/2}f(s^{-1} R^{-1}(t-x))$
with $R\in SO(d)$. Then $\pi_1$ is irreducible.  For
rotation-invariant functions $g$ we have $\pi_1(x,s,R)g=\pi(x,s)g$, so
we may as well work with the reducible $\pi$.  The representation
coefficients of $\pi$ are nothing 
but the continuous wavelet transform, which 
% Recall that  the continuous wavelet transform of $f\in L^2(\mathbb{R}^d)$ with respect to a non-zero wavelet
% $g\in L^2(\mathbb{R}^d)$
 is defined by
\[W_g f(x,s) = \langle f, \pi(x,s)
g\rangle=s^{-d/2}\int_{\mathbb{R}^d} f(t)
\overline{g\left(\frac{t-x}{s}\right)}\mathrm{d}t \]
for $f,g \in L^2(\mathbb{R}^d), g\neq 0$. 

We first identify the coorbit spaces with respect to the
representation $\pi $ of $\mathcal{G}_A$. Let $1\leq p,q< \infty$ and
$w(x,s) = s^{-\sigma}$ for $\sigma\in \mathbb{R}$.
% In the sequel, we work with the submultiplicative weight function
% $w(x,s)=s^{-\sigma}$ with $\sigma\in \mathbb{R}$. 
% Below we state the main result of \cite{Gr88}, describing the interpretation of  Besov spaces
% as coorbit space with respect to the affine group.
  The mixed norm space
$L^{p,q}_\sigma(\mathcal{G}_{\mathrm{A}})$ is defined by the   norm
\[\|F\;|L^{p,q}_\sigma(\mathcal{G}_{\mathrm{A}})\|=\left(\int_{0}^\infty \left(\int_{\mathbb{R}^d} |F(x,s)|^p
\;\mathrm{d}x\right)^{q/p}s^{-\sigma
q}\frac{\mathrm{d}s}{s^{d+1}}\right)^{1/q}\]
with the usual modifications when $p=\infty $ or$q=\infty$.

Recall the classical definition of the  homogenous Besov spaces.
Let  $\varphi\in \mathcal{S}(\mathbb{R}^d)$ with  $\mathrm{supp} (\varphi)\subset \{y\in\mathbb{R}^d:\;|y|<2\}$
and $\varphi(x)=1$ if $|x|\leq 1$ and set
$\varphi_j(x)=\varphi(2^{-j}x)-\varphi(2^{-j+1}x)$, $j\in \mathbb{Z}$.
For  $1\leq p,q\leq \infty, \; \sigma\in \mathbb{R}$,
the homogenous Besov space  $\dot{B}_{pq}^{\sigma}(\mathbb{R}^d)$ is
the set of all tempered distribution modulo polynomials
$f\in 	\mathcal{S}'/ \mathcal{P}(\mathbb{R}^d)$ such that
		\begin{align} 				
		\big\|f\;|\dot{B}_{pq}^{\sigma}(\mathbb{R}^d)\big\|=
		\left(\sum_{j\in \mathbb{Z}}2^{j\sigma
                    q}\big\|\mathcal{F}^{-1} (\varphi_j\widehat{f})|
		L^p(\mathbb{R}^d)\big\|^q \right)^{1/q}
		\end{align}
	is finite, with the usual modification for  $q=\infty$. 	
A result of Triebel \cite{Tr88} yields the equivalent norm on $\dot{B}_{pq}^{\sigma}(\mathbb{R}^d)$:
\[\left(\int_0^\infty s^{-q(\sigma+d/2-d/q)}\|W_g f(\cdot,s)\;| L^p(\mathbb{R}^d)\|^q \frac{\mathrm{d}s}{s^{d+1}}\right)^{1/q}
=\big\|W_gf\;|L^{p,q}_{\sigma+d/2-d/q}(\mathcal{G}_{\mathrm{A}})\big\|.\]
Triebel's result reveals that  the  homogenous Besov spaces coincide
with some  coorbits spaces of the affine group $\mathcal{G}_A$.   More precisely,
\[\dot{B}_{pq}^{\sigma}(\mathbb{R}^d)=\mathrm{Co}(L^{p,q}_{\sigma+d/2-d/q}(\mathcal{G}_{\mathrm{A}})).\] 

Next we compare classical molecules as in \cite{FrJaWe91_Book} and the 
coorbit  molecules  according to Definition \ref{definition_molecules}. 
Let us start by describing the classical molecules.
For  $k=(k_1,\ldots,k_d)\in \mathbb{Z}^d$ and  $j\in \mathbb{Z}$, a dyadic cube is given by
$Q=Q_{jk}=\{(x_1,\ldots,x_d)\in \mathbb{R}^d: k_i\leq 2^j x_i <k_i+1\}$.
Its left corner is  $x_Q=x_{Q_{jk}}=2^{-j}k$, its side length  $\ell(Q)=\ell(Q_{jk})=2^{-j}$, and its volume is $|Q|=2^{-jd}$.
 For $M,N\in \{-1,0,1,2,\ldots\}$ a classical smooth $(M,N)$-molecule associated to a dyadic cube $Q$ is a function $m_Q$
 satisfying the estimates
\begin{align}\label{condition_decay}|\mathrm{D}^{\alpha}m_Q(x)|\leq |Q|^{-1/2-|\alpha|/d}\left\{1+\frac{|x-x_Q|}{l(Q)}\right\}^{-M} \;\; \text{ for }
                 \;\;  |\alpha| \leq M,\; x\in \mathbb{R}^d.
                 \end{align}
and the moment conditions
\begin{align}\label{condition_moments}
\int_{\mathbb{R}^d}x^{\beta}m_Q(x)\;\mathrm{d}x=0 \;\; \text{ for } \;\; |\beta|\leq N.
\end{align}

This  notion of a classical smooth molecule goes back to 
\cite{FrJa85}, see also~\cite{FJ90,FHJW,FrJaWe91_Book}.%  and  the
                                %  references given there. 
[The atoms in classical analysis are defined similarly with the decay condition \eqref{condition_decay}
being replaced by an appropriate support condition.]
% In contrast to the definition of classical  atoms we require no
% conditions on support of molecules. ****
%The condition \eqref{condition_decay} is void if $M=-1$. Similarly, \eqref{condition_moments} is void if $N=-1$.
To understand how the conditions~\eqref{condition_decay}
and~\eqref{condition_moments} can be expressed by the wavelet
transform, we note that  the decay condition \eqref{condition_decay}
can be rephrased as 
\begin{align}\label{condition_decay_2}
|\mathrm{D}^{\alpha}m_{Q_{jk}}(x)|\leq 2^{jd/2+j|\alpha|}(1+|2^j x-k|)^{-M} \;\; \text{ for }
                 \;\;  |\alpha| \leq M. 
  \end{align} 
and                 
the moment conditions \eqref{condition_moments} as 
\[\widehat{m_{Q_{jk}}}(\xi)\leq C_n\ |\xi|^n, \quad |\xi|\rightarrow
0,\;\text{ for all }  n\leq N.\] 
The next proposition describes the decay of wavelet transform of the classical molecules.

\begin{proposition}\label{wavemom}
Let $g$ and $f$ satisfy the conditions \eqref{condition_decay_2} with $j=k=0$ and \eqref{condition_moments}.
Then there are numbers $\alpha, \beta, \gamma \in \mathbb{N}$ depending only on $M,N$ and a constant  $C_{\alpha, \beta, \gamma}>0$
such that
\begin{align}\label{wavelet_estimate}
|W_g f(x,s)| \leq C_{\alpha, \beta, \gamma} s^{\alpha} (1+s)^{-\beta} (1+|x|)^{-\gamma}.
\end{align}
\end{proposition}
By improving the quality of the window we can achieve a stronger result. 
\begin{remark}\upshape \label{remark_Besov_infty}
In particular, if $g \in \mathcal{S}(\mathbb{R}^d)$ has
all  moments vanishing, % . Further suppose that $g\in
% \mathcal{S}^\infty(\mathbb{R}^d)$, 
then  for every  $\alpha, \beta, \gamma\in \mathbb{N}$ there is a constant $C_{\alpha, \beta, \gamma}>0$ such that
\begin{align}\label{wavelet_estimate_infty}
|W_g g(x,s)| \leq C_{\alpha, \beta, \gamma} s^{\alpha} (1+s)^{-\beta} (1+|x|)^{-\gamma}.
\end{align}
\end{remark}
Proposition~\ref{wavemom} is due to Holschneider~\cite{Ho95_Book},
though the result is somewhat hidden in the proofs  of his Theorems
11.0.2,   12.0.1, and  19.0.1. (\cite{Ho95_Book} uses a different
normalization of the wavelet transform and treats the  dimensions
$d=1$ and $d>1$ separately).%  in

% We simply combine Theorems  from there. Remark that these results are shown only in one dimension.
% For the discussion of the extension of a needed results to higher dimensions the reader should consult Section 30 of this monograph.
% Further, we note that the wavelet transform  used in \cite{Ho95_Book} differs from the wavelet transform in our paper
% by the factor $s^{d/2}$.

The next proposition clarifies the relation between classical
molecules  and  coorbit molecules  for the affine group. 
\begin{proposition}\label{proposition_Besov}
Fix a weight function $w$ on $\mathcal{G}_A$. %  an envelope function
                                %  $H\in W^R(L^\infty,
                                %  L^1_w)(\mathcal{G}_{\mathrm{A}})$.
Then for $M,N$ sufficiently large, every set
of $(M,N)$-molecules $(m_{Q_{jk}})$ is a set of coorbit molecules in the sense of Definition \ref{definition_molecules}.
\end{proposition}

\begin{proof}
Note that the dyadic cube $Q_{jk}=2^{-j}(k+[0,1]^d)$ is attached to the point $x_{jk}=(2^{-j}k,2^{-j})\in \mathcal{G}_{\mathrm{A}}$.
To show that $(m_{Q_{jk}})$ is a set of coorbit molecules with
envelope function $H$, we need to show that 
\[|W_g m_{Q_{jk}}(x,s)|\leq H((2^{-j}k,2^{-j})^{-1}(x,s)) =H(2^jx-k,2^j s).\]
In view of estimate~\eqref{wavelet_estimate}, the natural candidate
for an envelope  $H$ is the function
\[H(x,s)= s^{\alpha} (1+s)^{-\beta} (1+|x|)^{-\gamma} \]
with $ \alpha, \beta, \gamma\in \mathbb{N}$ depending on $M,N$. Our 
task is  to show that $H\in  W^R(L^\infty,L_w^1)$. We must first
estimate the local maximum function $F^R_\sharp $ of $H$. 
 We set $U=B(0,a)\times [b^{-1},b]$ with $a>0$ and $b>1$.
Then 
\begin{align*}
F_\sharp^R (x,s)&= \sup _{(u,v) \in U\inv (x,s)\inv } |H(u,v)| 
=\sup _{(y,r)\in U}  H((y,r)\inv ((x,s)\inv )\\
% *****\|(R_{(x,s)} \chi_U^\vee )H|\; L^{\infty}\|=\|(L_{(x,s)} \chi_U )H^{\vee}|\; L^{\infty}\|
% =\| \chi_U ( L_{(x,s)^{-1}}H^{\vee})|\; L^{\infty}\|\\&=\sup_{(y,r)\in U} \left|H^{\vee}((x,s)(y,r))\right|
% =\sup_{(y,r)\in U} \left|H^{\vee}(x+sy,sr)\right|\\
&=\sup_{(y,r)\in U} \left|H\left(- \frac{x+sy}{sr}, \frac{1}{sr}\right)\right|=\sup_{(y,r)\in U}
\left(\frac{1}{sr}\right)^{\alpha}\left(1+\frac{1}{sr}\right)^{-\beta}\left(1+\frac{|x+sy|}{sr}\right)^{-\gamma}\\
&=\sup_{y\in B(0,a)} \sup_{r\in [b^{-1},b]}(sr)^{-\alpha+\beta}(1+sr)^{-\beta} \left(1+|x/sr+ y/r|\right)^{-\gamma}\\
&\leq C_b \sup_{y\in B(0,a)}  s^{-\alpha+\beta}(1+s)^{-\beta}
\left(1+|x/s+ y|\right)^{-\gamma} \\
&\leq  C_{ab}  s^{-\alpha+\beta}(1+s)^{-\beta} \left(1+|x/s|\right)^{-\gamma}.
\end{align*}
In the last estimate the moderateness of the weight
$(1+|\cdot|)^{-\gamma}$ has been used.

The $W^R(L^\infty , L^1_w)$-norm of $H$ is then 
\begin{eqnarray*}
  \|H| W^R(L^\infty , L^1_w)\|&=& \int_{\mathbb{R}^d}\int_{0}^{\infty}
  F_\sharp^R (x,s)  s^{-\sigma} \mathrm{d}x \frac{\mathrm{d}s}{s^{d+1}}\;
   \\
&\leq &  C_{ab}   \int _0 ^\infty \int _{\mathbb{R}}
  s^{-\alpha+\beta}(1+s)^{-\beta} \left(1+|x/s|\right)^{-\gamma} \,
  \mathrm{d}x \frac{\mathrm{d}s}{s^{d+1}} \, ,
\end{eqnarray*}
and this integral converges, if $\gamma >d$ and $\beta > \alpha
+\sigma >0$. 
\end{proof}

Finally  we  apply Theorem~\ref{main_theorem}  to study the
boundedness of Hilbert transform on 
homogenous Besov spaces.
Recall that the Hilbert transform $\mathbf{H}$ of a function $f$ is given by
\[\mathbf{H}f(x)=\lim_{\varepsilon\rightarrow 0}\frac{1}{\pi}
\int_{|t|\geq \varepsilon} \frac{f(x-t)}{t}\mathrm{d}t,\] 
provided that the limit exists. The boundedness of $\mathbf{H}$ on
Besov spaces   follows from Fourier multiplier theorems
for Besov spaces, e.g.,  \cite{Tr83}, or from Lemari\'e's work on
Calder\`on-Zygmund operators on Besov spaces~\cite{Lem85}.  Here we show that it is an immediate consequence of Theorem \ref{main_theorem}.
\begin{proposition}
Let $1\leq p,q\leq \infty$ and $\sigma\in \mathbb{R}$. Then the
Hilbert transform is  bounded  on 
$\dot{B}_{pq}^{\sigma}(\mathbb{R}^d)$.
\end{proposition}
\begin{proof}
We choose a basis function
$g\in \mathcal{S}(\mathbb{R})$ such  that $\mathrm{supp}\, \hat{g}
\subseteq \{ \omega \in \mathbb{R}: 1/2 \leq |\omega | \leq 2\}$ and 
$\{ \pi (2^{-j}k,2^{-j}g : 
j,k\in \mathbb{Z}\}$ is a Banach frame  for
$\dot{B}_{pq}^{\sigma}(\mathbb{R}^d)$.
Since the Hilbert transform commutes with all translations $T_x$ and
 dilations $D_s$, i.e., 
$\mathbf{H}(T_xD_s)f(t)=T_xD_s (\mathbf{H}f)(t)$ $\mathbf{H}$ maps the frame
$\pi(2^{-j}k,2^{-j})  g$  into atoms $\pi (2^{-j}k,2^{-j}) \mathbf{H}g$. 
 Therefore it suffices  to prove that 
 $W_g\mathbf{H}g \in W^R(L^\infty, L^1_w)(\mathcal{G}_A)$ where $w(x,s) =
 s^{-\sigma +d/2-d/q}$, then $K=|W_g\mathbf{H}g|$
 serves as an envelope for which \eqref{formula_definition_molecules}
 holds and $\mathbf{H}(\pi(2^{-j}k,2^{-j}))  g, j,k \in \mathbb{Z}$ is a set of
 molecules.  
 Since both $g$ and $\mathbf{H}g$ are in $\mathcal{S}$ with all vanishing
 moments, estimate~\eqref{wavelet_estimate} and the proof of
 Proposition \ref{proposition_Besov} show that    $W_g\mathbf{H}g
 \in W^R(L^\infty, L^1_w)(\mathcal{G}_A)$. 
% and taking into account the estimate \eqref{wavelet_estimate_infty} finishes the proof.
\end{proof}

\bibliographystyle{abbrv}

\end{document}